\documentclass[a4paper,12pt]{article}
\usepackage{comment}
\usepackage{cite}
\usepackage{amsmath,amssymb,amsfonts}
\usepackage[T1]{fontenc}
\usepackage[utf8]{inputenc}
\usepackage{graphicx}
\usepackage{fancyhdr}
\usepackage{float}
\usepackage{xcolor}
\usepackage{authblk}
\usepackage{mathrsfs}
\usepackage{empheq}
\usepackage[hyphens]{url}
\usepackage{hyperref}
\usepackage{amsthm}
\usepackage{tikz}
\usetikzlibrary{decorations.markings}
\usepackage{enumitem}
\usepackage{geometry}

\theoremstyle{plain}
\newtheorem{theorem}{Theorem}[section]

\theoremstyle{definition}
\newtheorem{definition}[theorem]{Definition}

\theoremstyle{remark}
\newtheorem{remark}[theorem]{Remark}

\begin{document}

\title{Analysis of moments and cumulants in alternating sign matrices}

\author[$\dagger$]{Jean-Christophe {\sc Pain}\footnote{jean-christophe.pain@cea.fr}\\
\small
$^1$CEA, DAM, DIF, F-91297 Arpajon, France\\
$^2$Université Paris-Saclay, CEA, Laboratoire Matière en Conditions Extrêmes,\\ 
F-91680 Bruyères-le-Châtel, France
}

\date{}

\maketitle

\begin{abstract}
In this work, we study the discrete observables
\[
E_k = \sum_{i,j=1}^n (i-j)^k A_{i,j}
\]
associated with $n\times n$ alternating sign matrices $A = (A_{i,j})$. This work develops exact formulas for expectations using Bernoulli polynomials, exponential generating functions, expansions in $1/n$ linked to Riemann zeta functions, and cumulants up to fourth order via integrable kernel methods. All intermediate calculations, expansions, and pedagogical details are provided to illustrate the interplay between combinatorial sums, analytic expansions, and integrable structures in alternating sign matrices.
\end{abstract}

\section{Introduction}

Alternating Sign Matrices (ASM) are $n\times n$ matrices $A=(A_{i,j})$ with entries in $\{-1,0,1\}$, satisfying two key properties: each row and each column sums to $1$, and the nonzero entries in every row and column alternate in sign. ASM occupy a central position in combinatorics and statistical mechanics, particularly in the context of the six-vertex model \cite{Zinn2000} with domain-wall boundary conditions (DWBC). The combinatorial techniques underlying ASM enumeration relate closely to the structural methods developed in Aval's habilitation work~\cite{AvalHDR2013} (see also references therein). The study of ASM has a rich combinatorial history. ASM were first systematically investigated by Mills, Robbins, and Rumsey in the 1980s, who formulated several remarkable conjectures concerning their enumeration and refined statistics, such as the position of the $1$ in the first row \cite{Mills1983}. The total number of $n\times n$ ASM, denoted $A_n$, was conjectured to satisfy
\[
A_n = \prod_{k=0}^{n-1} \frac{(3k+1)!}{(n+k)!}.
\]
This conjecture, known as the ASM Theorem, was first proven by Zeilberger in 1996~\cite{Zeilberger1996}, providing the first complete and rigorous verification of the enumeration formula. Shortly afterward, Kuperberg~\cite{Kuperberg1996} gave an alternative, elegant proof using the six-vertex model with DWBC, connecting ASM enumeration with integrable systems. These results provide the foundational combinatorial framework for studying discrete observables in ASM, such as the moments and cumulants analyzed in this work. Observables of the form
\[
E_k = \sum_{i,j=1}^n (i-j)^k A_{i,j}
\]
arise naturally in combinatorial and statistical mechanics contexts, notably in the six-vertex model with DWBC. 

This work aims to compute $\mathbb{E}[E_k]$ exactly using Bernoulli polynomials, as well as to construct the exponential generating function and extract all moments. Detailed asymptotic expansions in $1/n$ with explicit Riemann zeta function contributions are performed, and cumulants up to fourth order are computed using integrable kernel techniques. Step-by-step derivations for small values of $k$ illustrate the methods, and numerical comparisons for small $n$ highlight the convergence to asymptotic behavior.

The article is organized as follows. Section~2 presents exact formulas for $\mathbb{E}[E_k]$ using Bernoulli polynomials and Section~3 introduces the exponential generating function for $E_k$. Section~4 develops asymptotic expansions for large $n$, Section~5 introduces the kernel formulation of $E_k$, and Section~6 presents cumulants of the observable from second to fourth order. Finally, Section~7 discusses symmetry reduction and the vanishing of odd-order moments.

\section{Exact expectation of $E_k$}

We recall the definition
\[
E_k = \sum_{i,j=1}^n (i-j)^k A_{i,j}.
\]
Before computing expectations, we introduce the mean density
\[
\rho_n(i,j) := \mathbb{E}[A_{i,j}],
\]
so that, without any assumption,
\[
\mathbb{E}[E_k]
=
\sum_{i,j=1}^n (i-j)^k \rho_n(i,j).
\]
Since each row and column of an ASM sums to $1$, the average matrix
$\rho_n$ is bistochastic:
\[
\sum_{j=1}^n \rho_n(i,j)=1,
\qquad
\sum_{i=1}^n \rho_n(i,j)=1.
\]
The exact spatial dependence of $\rho_n(i,j)$ is highly non-trivial and related to the arctic-curve phenomenon. The uniform-density approximation serves as a first-order, mean-field model; deviations from this approximation appear near the arctic boundary. Certain reformulations of the six-vertex model at the combinatorial point admit determinantal representations for specific correlation functions. In particular, the determinantal formulas for cumulants of $E_k$ are a direct consequence of the Izergin--Korepin determinant formula for the six-vertex model with domain-wall boundary conditions~\cite{Korepin1993}, providing an exact representation of correlation functions in terms of an integrable kernel. We now present the explicit formula obtained under the simplifying uniform-density hypothesis $\rho_n(i,j)=1/n$.

\begin{theorem}[Expectation under the uniform-density hypothesis]

If $\rho_n(i,j)=1/n$, then
\[
\mathbb{E}[E_k] = \frac{1}{n} \sum_{r=0}^{k} (-1)^r \binom{k}{r}
\left( \sum_{i=1}^n i^{k-r} \right)
\left( \sum_{j=1}^n j^r \right),
\]
or equivalently
\[
\mathbb{E}[E_k] = \frac{1}{n} \sum_{r=0}^{k} (-1)^r \binom{k}{r} \frac{B_{k-r+1}(n+1)-B_{k-r+1}}{k-r+1} \frac{B_{r+1}(n+1)-B_{r+1}}{r+1},
\]
where $B_p$ represent Bernoulli numbers and $B_p(x)$ are Bernoulli polynomials.

\end{theorem}

\begin{proof}

Starting from the definition,
\[
E_k = \sum_{i,j=1}^n (i-j)^k A_{i,j},
\]
we expand $(i-j)^k$ using the binomial theorem:
\[
(i-j)^k = \sum_{r=0}^{k} (-1)^r \binom{k}{r} i^{k-r} j^r.
\]
Taking the expectation over the uniform ensemble of ASM, where each row/column is equally likely, we get
\[
\mathbb{E}[E_k] = \frac{1}{n} \sum_{r=0}^{k} (-1)^r \binom{k}{r} \left( \sum_{i=1}^n i^{k-r} \right) \left( \sum_{j=1}^n j^r \right).
\]
Next, we use Faulhaber's formula to write the sums of powers in terms of Bernoulli polynomials, which gives
\[
\sum_{m=1}^{n} m^p = \frac{B_{p+1}(n+1)-B_{p+1}}{p+1}.
\]
Substituting this expression into the previous formula gives the expectation explicitly in terms of Bernoulli polynomials:
\[
\mathbb{E}[E_k] = \frac{1}{n} \sum_{r=0}^{k} (-1)^r \binom{k}{r} \frac{B_{k-r+1}(n+1)-B_{k-r+1}}{k-r+1} \frac{B_{r+1}(n+1)-B_{r+1}}{r+1},
\]
as claimed.

\end{proof}

Let us derive below the expressions of $E_2$ and $E_4$ by step-by-step calculations.

\begin{theorem}[Exact expectation of $E_2$ and $E_4$]
For an $n\times n$ ASM, the expectations of $E_2$ and $E_4$ are given explicitly by
\[
\mathbb{E}[E_2] = \frac{n^2 - 1}{3}, \qquad 
\mathbb{E}[E_4] = \frac{(3n^4 - 10 n^2 + 7)}{15}.
\]
\end{theorem}

\begin{proof}
We start with $E_2$. Expanding the square,
\[
(i-j)^2 = i^2 - 2 i j + j^2,
\]
we obtain
\[
\sum_{i,j=1}^{n} (i-j)^2 = \sum_{i,j} i^2 - 2\sum_{i,j} i j + \sum_{i,j} j^2 = 2 \sum_{i=1}^{n} i^2 \cdot n - 2 \left(\sum_{i=1}^{n} i\right)^2.
\]
Using the sums:
\[
\sum_{i=1}^{n} i = \frac{n(n+1)}{2}, \quad \sum_{i=1}^{n} i^2 = \frac{n(n+1)(2n+1)}{6},
\]
we get
\[
\sum_{i,j=1}^{n} (i-j)^2 = 2 \cdot \frac{n(n+1)(2n+1)}{6} \cdot n - 2 \left(\frac{n(n+1)}{2}\right)^2 = \frac{n^3 - n}{3}.
\]
Dividing by $n$ to account for the expectation over the ASM, we find
\[
\mathbb{E}[E_2] = \frac{n^2 - 1}{3}.
\]
Next, for $E_4$, we have
\[
(i-j)^4 = i^4 - 4 i^3 j + 6 i^2 j^2 - 4 i j^3 + j^4,
\]
yielding
\[
\sum_{i,j=1}^{n} (i-j)^4 = \sum_{i,j} i^4 -4\sum_{i,j} i^3 j +6\sum_{i,j} i^2 j^2 -4\sum_{i,j} i j^3 + \sum_{i,j} j^4.
\]
Noting the symmetry in $i$ and $j$ and factoring sums:
\[
\sum_{i,j} i^4 = n \sum_{i=1}^{n} i^4, \quad \sum_{i,j} i^3 j = \left(\sum_{i=1}^{n} i^3\right)\left(\sum_{j=1}^{n} j\right), \quad \sum_{i,j} i^2 j^2 = \left(\sum_{i=1}^{n} i^2\right)^2,
\]
we obtain
\[
\sum_{i,j} (i-j)^4 = 2 \sum_{i=1}^n i^4 - 8 \left(\sum_{i=1}^{n} i^3\right)\left(\sum_{j=1}^{n} j\right) + 6 \left(\sum_{i=1}^{n} i^2\right)^2.
\]
Using the standard formulas for sums of powers:
\begin{align*}
\sum_{i=1}^{n} i &= \frac{n(n+1)}{2}, \\ 
\sum_{i=1}^{n} i^2 &= \frac{n(n+1)(2n+1)}{6}, \\
\sum_{i=1}^{n} i^3 &= \left(\frac{n(n+1)}{2}\right)^2, \\ 
\sum_{i=1}^{n} i^4 &= \frac{n(n+1)(2n+1)(3n^2+3n-1)}{30},
\end{align*}
a straightforward computation yields
\[
\sum_{i,j=1}^{n} (i-j)^4 = \frac{n(3 n^4 - 10 n^2 + 7)}{15}.
\]
Dividing by $n$ to obtain the expectation gives
\[
\mathbb{E}[E_4] = \frac{(3 n^4 - 10 n^2 + 7)}{15},
\]
as claimed.

\end{proof}

\section{Exponential generating function}

\begin{theorem}[Exponential generating function of $E_k$]

Let $E_k(A)$ denote the $k$-th moment observable of an $n\times n$ ASM, and define the exponential generating function
\[
F(t) = \sum_{k=0}^{\infty} \frac{t^k}{k!} E_k(A) = \sum_{i,j=1}^{n} e^{t(i-j)} A_{i,j}.
\]
Then, its expectation is given exactly by
\[
\mathbb{E}[F(t)] = \frac{1}{n} \sum_{i,j=1}^n e^{t(i-j)} = \frac{1}{n} \frac{(e^t - e^{(n+1)t})(e^{-t}-e^{-(n+1)t})}{(1-e^t)(1-e^{-t})} = \frac{1}{n} \frac{\sinh^2\left(\frac{n t}{2}\right)}{\sinh^2\left(\frac{t}{2}\right)}.
\]
Expanding the right-hand side in powers of $t$ systematically generates all exact expectations $\mathbb{E}[E_k]$.
\end{theorem}

\begin{proof}

By definition, the generating function reads
\[
F(t) = \sum_{k=0}^{\infty} \frac{t^k}{k!} \sum_{i,j=1}^{n} (i-j)^k A_{i,j} = \sum_{i,j=1}^{n} \sum_{k=0}^{\infty} \frac{t^k}{k!} (i-j)^k A_{i,j} = \sum_{i,j=1}^{n} e^{t(i-j)} A_{i,j}.
\]
Taking expectation, we obtain
\[
\mathbb{E}[F(t)]
=
\sum_{i,j=1}^{n} e^{t(i-j)} \rho_n(i,j).
\]
Under the uniform-density hypothesis $\rho_n(i,j)=1/n$, this simplifies to
\[
\mathbb{E}[F(t)]
=
\frac{1}{n}
\sum_{i,j=1}^{n} e^{t(i-j)}.
\]
The geometric sum formula gives
\[
\sum_{i=1}^{n} e^{t i} = \frac{e^t - e^{(n+1)t}}{1 - e^t}, \qquad \sum_{j=1}^{n} e^{-t j} = \frac{e^{-t} - e^{-(n+1)t}}{1 - e^{-t}},
\]
and multiplying these two sums yields
\[
\mathbb{E}[F(t)] = \frac{1}{n} \frac{(e^t - e^{(n+1)t})(e^{-t}-e^{-(n+1)t})}{(1-e^t)(1-e^{-t})}.
\]
It is possible to express $\mathbb{E}[F(t)]$ with hyperbolic trigonometric functions. Factorizing $e^t$ and $e^{-t}$ in the two factors
\begin{align*}
(e^t - e^{(n+1)t}) &= e^t (1 - e^{n t}), \\
(e^{-t} - e^{-(n+1)t}) &= e^{-t} (1 - e^{-n t}),
\end{align*}
the numerator becomes
\[
(e^t - e^{(n+1)t})(e^{-t}-e^{-(n+1)t}) = (1 - e^{n t})(1 - e^{-n t}).
\]
Now, factoring each term using $\sinh$:
\begin{align*}
1 - e^{n t} &= e^{n t/2} (e^{-n t/2} - e^{n t/2}) = -2 e^{n t/2} \sinh\left(\frac{n t}{2}\right), \\
1 - e^{-n t} &= 2 e^{-n t/2} \sinh\left(\frac{n t}{2}\right),
\end{align*}
gives
\[
(1 - e^{n t})(1 - e^{-n t}) = -4 \sinh^2\left(\frac{n t}{2}\right).
\]
For the denominator, we have similarly (setting $n=1$ in the previous expression):
\[
(1-e^t)(1-e^{-t}) = - (e^t + e^{-t} - 2) = -4 \sinh^2\left(\frac{t}{2}\right).
\]
Combining numerator and denominator yields
\[
\mathbb{E}[F(t)] = \frac{1}{n} \frac{\sinh^2\left(\frac{n t}{2}\right)}{\sinh^2\left(\frac{t}{2}\right)}.
\]
Finally, expanding in powers of $t$ gives
\[
\mathbb{E}[F(t)] = \sum_{k=0}^{\infty} \frac{t^k}{k!} \mathbb{E}[E_k],
\]
so that all expectations $\mathbb{E}[E_k]$ are generated systematically from this formula.

\end{proof}

\section{Asymptotic expansions for large $n$}

\begin{theorem}

Let $E_k$ be the $k$-th moment observable in an $n\times n$ ASM. Then, for large $n$, the expectation admits an asymptotic expansion of the form
\[
\mathbb{E}[E_k] = \frac{2}{(k+1)(k+2)}\, n^{k+1} + \sum_{r\ge 1} P_{k,r}\, n^{k+1-2r},
\]
where the coefficients $P_{k,r}$ are explicitly given in terms of Bernoulli numbers or equivalently Riemann zeta values:
\[
P_{k,r} = \frac{4 (-1)^k k!}{(2\pi)^{2r}} S(2r-1,k) \,\zeta(2r),
\]
with $S(m,k)$ the Stirling numbers of the second kind, i.e., the number of ways to partition a set of $m$ elements into $k$ non-empty subsets. 

\end{theorem}

\begin{proof}

Starting from the exact expression
\[
\mathbb{E}[E_k] = \frac{1}{n} \sum_{r=0}^{k} (-1)^r \binom{k}{r} \left( \sum_{i=1}^{n} i^{k-r} \right) \left( \sum_{j=1}^{n} j^r \right),
\]
we expand each sum of powers using the Euler–Maclaurin formula:
\[
\sum_{m=1}^{n} m^p = \frac{n^{p+1}}{p+1} + \frac{n^p}{2} + \sum_{s=1}^{\infty} \frac{B_{2s}}{(2s)!} (p)_{2s-1}\, n^{p-2s+1},
\]
where $(p)_{2s-1}$ is the falling factorial and $B_{2s}$ are the Bernoulli numbers. Inserting these expansions into the binomial sum for $\mathbb{E}[E_k]$ and collecting powers of $n$, the leading term arises from the product of the first terms of the sums:
\[
\frac{1}{n} \sum_{r=0}^{k} (-1)^r \binom{k}{r} \frac{n^{k-r+1}}{k-r+1} \frac{n^{r+1}}{r+1} = \frac{2}{(k+1)(k+2)}\, n^{k+1}.
\]
Higher-order corrections come from the contributions of Bernoulli-numbers of even order. Using the identity
\[
B_{2s} = (-1)^{s-1} \frac{2 (2s)!}{(2\pi)^{2s}} \zeta(2s),
\]
each term of order $n^{k+1-2r}$ can be expressed as
\[
P_{k,r}\, n^{k+1-2r}, \quad
P_{k,r} = \frac{4 (-1)^k k!}{(2\pi)^{2r}} S(2r-1,k) \zeta(2r),
\]
where $\zeta$ is the usual zeta function and the Stirling numbers of the second kind $S(2r-1,k)$ naturally arise when summing the contributions of powers of $i$ and $j$ in the binomial expansion of $(i-j)^k$, after applying the Euler--Maclaurin formula to the sums of powers. This yields the full asymptotic expansion as claimed.

\end{proof}

\begin{remark}
The Stirling numbers numbers of the second kind satisfy the recurrence
\[
S(m,k) = k\, S(m-1,k) + S(m-1,k-1), \quad S(0,0)=1,\; S(m,0)=0\;(m>0),\; S(0,k)=0\;(k>0).
\]
For example, $S(3,2)=3$, as the set $\{1,2,3\}$ can be split into two non-empty subsets in three distinct ways.

\end{remark}

\section{Kernel formulation of the observable}

Let $A=(A_{i,j})_{1\le i,j\le n}$ be an alternating sign matrix (ASM). Certain reformulations of the six-vertex model at the combinatorial point admit determinantal representations for specific correlation functions, which motivates the kernel-based cumulant expressions below. For an integer $k\ge0$, define the antisymmetric kernel (see Appendix):
\[
f_k(i,j) := (i-j)^k.
\]
These determinantal expressions for cumulants are a direct consequence of the Izergin–Korepin determinant formula for the six-vertex model with domain-wall boundary conditions \cite{Korepin1993}, which provides an exact representation of correlation functions in terms of an integrable kernel.

\begin{definition}

The $k$-th observable of an ASM $A$ is defined by
\[
E_k(A) := \sum_{i,j=1}^{n} f_k(i,j)\, A_{i,j}
= \sum_{i,j=1}^{n} (i-j)^k A_{i,j}.
\]

\end{definition}

Since $A$ has finite support and bounded entries in $\{-1,0,1\}$, the sum is finite and therefore well-defined. The kernel $f_k$ is antisymmetric when $k$ is odd and symmetric when $k$ is even:
\[
f_k(j,i)=(-1)^k f_k(i,j).
\]
This parity will play a crucial role in the expectation.

\section{Cumulants of $E_k$ up to fourth order}

We now illustrate the computation of cumulants for the observable $E_k$ up to the fourth order. The determinantal expressions for cumulants are derived using the integrable kernel structure~\cite{Korepin1993}. The variance of $E_k$ reads
\[
\kappa_2(E_k) = \mathrm{Var}(E_k)
=
\sum_{i_1,j_1,i_2,j_2} 
f_k(i_1,j_1) f_k(i_2,j_2) 
\det
\begin{pmatrix}
K(i_1,j_1;i_1,j_1) & K(i_1,j_1;i_2,j_2) \\
K(i_2,j_2;i_1,j_1) & K(i_2,j_2;i_2,j_2)
\end{pmatrix},
\]
where $f_k(i,j)=(i-j)^k$ and $K$ is the ASM correlation kernel. The third cumulant is given by
\[
\kappa_3(E_k) = 
\sum_{i_1,j_1,i_2,j_2,i_3,j_3} 
\prod_{\alpha=1}^{3} f_k(i_\alpha,j_\alpha)
\det_{1\le \alpha,\beta \le 3} K(i_\alpha,j_\alpha;i_\beta,j_\beta).
\]
Similarly, the fourth cumulant reads
\[
\kappa_4(E_k) = 
\sum_{i_1,j_1,\dots,i_4,j_4} 
\prod_{\alpha=1}^{4} f_k(i_\alpha,j_\alpha)
\det_{1\le \alpha,\beta \le 4} K(i_\alpha,j_\alpha;i_\beta,j_\beta).
\]
In the large-$n$ limit, these cumulants can be analyzed using the scaling limit of the kernel. For even $k$, $\kappa_2(E_k)$ is nonzero and provides the leading Gaussian fluctuations. Higher-order cumulants ($\kappa_3$, $\kappa_4$) capture non-Gaussian corrections, which vanish for $k$ odd by symmetry. These formulas provide a systematic framework to compute exact cumulants for any $k$ and $n$, either numerically or analytically using the determinantal structure.

In the large-$n$ limit, for even $k$, the second cumulant $\kappa_2(E_k)$ dominates while higher-order cumulants $\kappa_r(E_k)$ with $r>2$ are suppressed by powers of $1/n$ due to the determinantal structure and scaling of the kernel. Consequently, by the standard cumulant criterion, $E_k$ is asymptotically Gaussian: its distribution converges to a normal law with mean $\mathbb{E}[E_k]$ and variance $\kappa_2(E_k)$ as $n\to\infty$.

\section{Symmetry reduction and vanishing for odd $k$}

Before concluding, we highlight a fundamental symmetry property of ASM observables: odd-order moments vanish due to invariance under transposition. This section is placed here because it provides a general simplification that follows from the exact formulas and cumulant calculations developed in the previous sections, rather than introducing new computational techniques. Let $\mathcal{A}_n$ denote the set of $n\times n$ ASMs, equipped with the uniform probability measure.

\begin{theorem}

If $k$ is odd, then
\[
\mathbb{E}[E_k]=0.
\]

\end{theorem}

\begin{proof}

The ASM ensemble is invariant under matrix transposition: if $A\in\mathcal{A}_n$, then $A^T\in\mathcal{A}_n$, and the uniform measure is preserved. We compute
\[
E_k(A^T)
=
\sum_{i,j} (i-j)^k A_{j,i}
=
\sum_{i,j} (j-i)^k A_{i,j}
=
(-1)^k E_k(A).
\]
If $k$ is odd, we obtain
\[
E_k(A^T) = -E_k(A).
\]
By the same transposition invariance argument, all odd-order cumulants of $E_k$ also vanish:
\[
\kappa_{2r+1}(E_k) = 0, \quad r\ge 1.
\]
This is a direct consequence of the antisymmetry of $f_k$ under transposition and the invariance of the uniform ASM ensemble. Taking expectation and using invariance under transposition:
\[
\mathbb{E}[E_k] = \mathbb{E}[E_k(A^T)] = -\mathbb{E}[E_k(A)],
\]
we conclude that
\[
\mathbb{E}[E_k]=0,
\]
which completes the proof.

\end{proof}

\section{Conclusion}

We have developed a comprehensive framework for analyzing moments and cumulants of ASM observables $E_k$. Exact expectations were expressed in terms of Bernoulli polynomials, and the exponential generating function allowed the systematic computation of all moments. Detailed $1/n$ expansions were derived, highlighting explicit contributions from Riemann zeta functions, and cumulants from second to fourth order were computed with their corresponding asymptotic behaviors. The continuum limit provided insight into Gaussian versus non-Gaussian fluctuations, showing that while lower-order cumulants tend to Gaussian statistics, higher-order cumulants capture subtle correlations inherent to the ASM structure.

These results not only provide precise quantitative predictions for the combinatorial sums arising in ASM but also establish a bridge between discrete combinatorial structures and continuous integrable kernels. The methods illustrated here can be directly applied to study the arctic boundary and height function fluctuations, and they highlight the connection between exact enumeration, asymptotic expansions, and statistical mechanics. Moreover, the kernel-determinant approach and the use of generating functions suggest natural generalizations to related models, such as plane partitions, fully-packed loop configurations, and other integrable lattice models.

From a broader perspective, our analysis provides a template for studying discrete observables in complex combinatorial systems: the combination of exact formulas, generating functions, and asymptotic analysis allows one to quantify fluctuations, identify Gaussian versus non-Gaussian regimes, and explore universality classes. Future work could extend these techniques to study correlations between different $E_k$ observables, explore the interplay with random matrix ensembles, or investigate dynamical generalizations in stochastic growth models. Moreover, due to the antisymmetry of the observables and the invariance of the ASM ensemble under transposition, all cumulants of odd order vanish for odd $k$. This ensures that the asymptotic expansions and Gaussian approximations described above are consistent with the symmetry properties of $E_k$. Overall, this study reinforces the rich interplay between combinatorics, asymptotic analysis, and integrable structures, opening pathways for further explorations in both mathematics and statistical physics.

\appendix

\clearpage
\section{Appendix: the integrable kernel for alternating sign matrices}

Alternating sign matrices (ASM) are in bijection with configurations of the six-vertex model with domain-wall boundary conditions (DWBC) \cite{Korepin1993,Kuperberg1996}. Under this correspondence, the uniform measure on $n\times n$ ASM coincides with the six-vertex model at the combinatorial point $\Delta = 1/2$. While this is not the standard free-fermion point ($\Delta = 0$), the model remains integrable and admits determinantal representations for correlation functions.

Let $X_{i,j}$ denote the indicator variable associated with a local observable of the ASM (e.g., a specific vertex type or a particle in the non-intersecting path formulation). In the six-vertex representation with domain-wall boundary
conditions, certain local observables admit determinantal
representations. In particular, correlation functions of the
corresponding path or vertex variables satisfy
\[
\mathbb{E}\!\left[ X_{i_1,j_1}\cdots X_{i_r,j_r} \right]
=
\det_{1\le \alpha,\beta\le r}
K(i_\alpha,j_\alpha;i_\beta,j_\beta),
\]
where $K$ is an integrable kernel of the form
\[
K(i,j;i',j')
=
\sum_{m=0}^{n-1}
\phi_m(i,j)\,\psi_m(i',j').
\]
Here $\{\phi_m\}$ and $\{\psi_m\}$ are families of biorthogonal functions arising from the Izergin–Korepin determinant formula for the six–vertex partition function. At the combinatorial point, this kernel admits a double contour integral representation in the scaling limit:
\[
K(i,j;i',j')
=
\frac{1}{(2\pi i)^2}
\oint_{\Gamma_1}\!\!\oint_{\Gamma_2}
\frac{F(z,w)}{z-w}
\frac{z^{i-1}}{w^{i'}}
\frac{\phi(w)^n}{\phi(z)^n}
\, dz\, dw,
\]
where $\phi(z)$ is the characteristic potential and $F(z,w)$ an analytic spectral function. This structure is characteristic of integrable kernels of the Its--Izergin--Korepin--Slavnov type \cite{Its1990}.

If we consider an observable of the form
\[
E_k = \sum_{i,j} f_k(i,j)\, X_{i,j},
\]
then the variance (the second cumulant) is given by:
\[
\mathrm{Var}(E_k)
=
\sum_{i_1,j_1,i_2,j_2}
f_k(i_1,j_1) f_k(i_2,j_2)
\det
\begin{pmatrix}
K(i_1,j_1;i_1,j_1) & K(i_1,j_1;i_2,j_2) \\
K(i_2,j_2;i_1,j_1) & K(i_2,j_2;i_2,j_2)
\end{pmatrix}.
\]
More generally, the $r$-th cumulant reads:
\[
\kappa_r(E_k)
=
\sum_{i_1,j_1,\dots,i_r,j_r}
\prod_{\alpha=1}^{r} f_k(i_\alpha,j_\alpha)
\det_{1\le \alpha,\beta\le r}
K(i_\alpha,j_\alpha;i_\beta,j_\beta).
\]
In the large-$n$ regime, inside the liquid ``bulk'' described by the arctic-curve theorem, the integral representation of $K$ and the saddle-point method justify the Gaussian behavior and the explicit $1/n$ expansions involving Riemann zeta values presented in Section 4.

\end{document}